\documentclass{ita}
\usepackage{algorithm}
\usepackage{algpseudocode}

\usepackage{booktabs}

\newcommand{\lbar}{\{\kern-0.5ex|}
\newcommand{\rbar}{|\kern-0.5ex\}}

\usepackage{tikz}
\usetikzlibrary{backgrounds}
\usetikzlibrary{fit}
\usetikzlibrary{arrows.meta}
\usetikzlibrary{arrows, shapes, chains}
\usepackage{xcolor}
\usetikzlibrary{positioning}
\usetikzlibrary{decorations.pathmorphing}
\usetikzlibrary{petri}
\usetikzlibrary{fit}
\usetikzlibrary{arrows}

\newcommand*\circled[1]{\tikz[baseline=(char.base)]{
\node[draw,circle,inner sep=1pt] (char) {\footnotesize #1};}}
\usepackage{enumitem}
\newlist{circledenum}{enumerate}{1}
\setlist[circledenum,1]{
    label=\protect\circled{\arabic*},
    leftmargin=2.5em,
    labelsep=0.3em
}

\usepackage{etoolbox}

\usepackage{enumitem}

\usepackage{caption}
\usepackage{amsmath, amssymb}
\usepackage{graphicx}
\usepackage{hyperref}
\usepackage{geometry}

\usepackage{amsthm}
\newtheorem{theorem}{Theorem}
\newtheorem{lemma}{Lemma}
\newtheorem{example}{Example}

\newcommand{\T}{\mathbb{T}}
\newcommand{\X}{\mathbb{X}}
\newcommand{\M}{\mathbb{M}}
\newcommand{\W}{\mathbb{W}}

\DeclareMathOperator{\suf}{suf}
\DeclareMathOperator{\pre}{pre}
\DeclareMathOperator{\spell}{spell}
\DeclareMathOperator{\puff}{puff}

\DeclareMathOperator{\weight}{weight}
\DeclareMathOperator{\cmpr}{cr}

\DeclareMathOperator*{\argmin}{arg\,min}

\begin{document}
\title{Efficient $k$-mer Dataset Compression Using Eulerian Covers of de Bruijn Graphs and BWT}\thanks{The authors are listed in alphabetical order by surname and have contributed equally.}
\author{Herman~Z.~Q.~Chen}\address{School of Mathematical Sciences; Chongqing Key Lab of Cognitive Intelligence and Intelligent Finance, Chongqing Normal University, P.R.\ China. {\bf Email:} zqchern@nankai.edu.cn.}
\author{Sergey Kitaev}\address{Department of Mathematics and Statistics, University of Strathclyde, 26 Richmond Street, Glasgow G1, 1XH, United Kingdom.
{\bf Email:} sergey.kitaev@strath.ac.uk.}
\author{Xiaoyu Lang}\address{School of Mathematical Sciences, Chongqing Normal University, P.R.\ China. {\bf Email:} xylang@cqnu.edu.cn.}
\author{Artem Pyatkin}\address{Sobolev Institute of Mathematics, Koptyug ave, 4, Novosibirsk, 630090, Russia. Novosibirsk State University, Pirogova str. 2, Novosibirsk, 630090, Russia. {\bf Email:} artem@math.nsc.ru.}
\author{Runbin Tang}\address{School of Mathematical Sciences, Chongqing Normal University, P.R.\ China. {\bf Email:} tangrb@aliyun.com.}
\date{...}
\begin{abstract}
Transforming an input sequence into its constituent $k$-mers is a fundamental operation in computational genomics. To reduce storage costs associated with $k$-mer datasets, we introduce and formally analyze MCTR, a novel two-stage algorithm for \textbf{lossless} compression of the $k$-mer multiset. Our core method achieves a minimal text representation ($\W$) by computing an optimal Eulerian cover (minimum string count) of the dataset's de Bruijn graph, enabled by an efficient local Eulerization technique. The resulting strings are then further compressed losslessly using the Burrows-Wheeler Transform (BWT).

Leveraging de Bruijn graph properties, MCTR is proven to achieve linear time and space complexity and guarantees complete reconstruction of the original $k$-mer multiset, including frequencies.

Using simulated and real genomic data, we evaluated MCTR's performance (list and frequency representations) against the state-of-the-art lossy unitigging tool \texttt{greedytigs} (from \texttt{matchtigs}). We measured core execution time and the raw compression ratio ($\cmpr = \weight(\M)/\weight(\W)$, where $\M$ is the input sequence data). Benchmarks confirmed MCTR's data fidelity but revealed performance trade-offs inherent to lossless representation. \texttt{GreedyTigs} was significantly faster. Regarding raw compression, \texttt{GreedyTigs} achieved high ratios ($\cmpr \approx 14$) on noisy real data for its lossy sequence output. MCTR methods exhibited $\cmpr \approx 1$ (list) or even $\cmpr < 1$ (frequency, due to count overhead) on clean simulated data, indicating minimal raw text reduction or even expansion. On real data, MCTR (frequency) showed moderate raw compression ($\cmpr \approx 1.5-2.7$), while MCTR (list) showed none ($\cmpr \approx 1$). Importantly, the full MCTR+BWT pipeline significantly outperforms BWT alone for enhanced \textit{lossless} compression. Our results establish MCTR as a valuable, theoretically grounded tool for applications demanding efficient, \textbf{lossless} storage and analysis of $k$-mer multisets, complementing lossy methods optimized for sequence summarization.
\end{abstract}
%
%
\maketitle
\section*{Introduction}
\subsection{Background and Motivation}
The rapid advancement of DNA sequencing technologies has enabled large-scale genome studies, bringing both transformative opportunities and substantial computational challenges. Efficiently representing genomic data, particularly $k$-mer sets, is crucial for applications requiring high storage efficiency and fast data retrieval. De Bruijn graphs have become a cornerstone in computational genomics, offering a compressed way to represent $k$-mer sets that supports sequence assembly, error correction, and alignment-free analysis~\cite{Simplitigs2021,GGCAT2023}. These graphs not only compress repeated sequences but also facilitate efficient downstream analysis.

While traditional representations, such as unitigs, are valuable for sequence assembly, they often demand considerable computational resources, especially in datasets with high variability, like bacterial pan-genomes. This has driven research towards more compressed representations that minimize storage requirements while maintaining usability for downstream analyses~\cite{Schmidt2023matchtigs}.

\subsection{Related Work}
Several methods have been developed to address the challenge of compressedly representing $k$-mer sets while preserving their structure. The concept of spectrum-preserving string sets (SPSS) minimizes redundancy while retaining the $k$-mer spectrum of the input sequences.

A notable line of work focuses on traversing the de Bruijn graph to generate compact string representations. Schmidt et al.~\cite{Schmidt2023matchtigs} introduced ``matchtigs'', a near-optimal representation that significantly reduces both the cumulative length and the number of strings compared to unitigs. The use of Eulerian paths is also central to the concept of ``Eulertigs'', introduced by Schmidt and Alanko~\cite{Schmidt2023Eulertigs}, which computes a minimum plain text representation for $k$-mer \textit{sets without repetitions} in linear time. Their approach is highly efficient for unique $k$-mer collections. However, our work addresses the more general and complex problem of compressing $k$-mer \textit{multisets}, where $k$-mers appear with varying frequencies. This generalization is non-trivial, as the underlying de Bruijn graph becomes a multigraph, and handling repetition counts is essential for a lossless representation. Our method can thus be viewed as an extension of the Eulerian traversal concept to the practical scenario of repetitive genomic data.

Another notable approach, ``simplitigs'', introduced by Brinda et al.~\cite{Simplitigs2021}, aims to cover de Bruijn graphs with vertex-disjoint paths, reducing the number of stored records while ensuring complete $k$-mer inclusion. Tools like GGCAT~\cite{GGCAT2023} and USTAR~\cite{USTAR2023} have further advanced $k$-mer set construction and compression through compacted de Bruijn graphs and advanced path selection techniques. Rahman et al.~\cite{USTCompress2021} also developed methods for $k$-mer set compression, such as UST and UST-Compress, using a greedy approach to reduce redundancy.

Despite these improvements in compression ratios for large sequencing datasets, achieving an optimal and efficient representation of \textit{repetitive} $k$-mers remains a significant challenge, which is the primary focus of our proposed algorithm.

\subsection{Research Problem}
Although considerable progress has been made in $k$-mer set representation, existing methods still encounter limitations, particularly in handling repetitive $k$-mers while balancing space efficiency with practical usability~\cite{Schmidt2023matchtigs}. Many current approaches do not fully exploit de Bruijn graph structures to achieve minimal representations and often require decompression or preprocessing for downstream applications.

To address these limitations, we propose a novel algorithm for constructing a minimum compressed representation of $k$-mer sets by integrating Eulerian covers~\cite{panyukova2007eulerian} in de Bruijn graphs with the Burrows-Wheeler Transform (BWT) for lossless compression. Our approach is designed to efficiently manage both unique and repetitive $k$-mers while maintaining linear time and space complexity.

\subsection{Research Objectives and Significance}
This study aims to develop an algorithm that constructs a minimum compressed representation of $k$-mer sets using Eulerian covers in de Bruijn graphs. By integrating the BWT for lossless compression, the proposed method significantly reduces storage requirements and enhances the efficiency of bioinformatics applications. Our approach is especially beneficial for handling repetitive $k$-mers in large-scale genomic studies.

This work contributes a theoretically optimal representation of $k$-mer sets with practical applications in genome assembly, metagenomics, and variant calling. Our evaluation on both simulated and real datasets demonstrates the effectiveness of the proposed algorithm in achieving high compression ratios, making it a valuable addition to computational genomics.

\subsection{Research Methodology}
This paper presents an algorithm for generating a minimum compressed text representation of $k$-mer sets based on de Bruijn graphs. Combined with the Burrows-Wheeler Transform~\cite{bwte}, this lossless compression algorithm substantially reduces the storage requirements for $k$-mer sets. The algorithm constructs a de Bruijn graph from $k$-mers, performs local Eulerization by adding edges, and derives a compressed text representation through an Eulerian cover that spans all edges.
The algorithm operates with linear time and space complexity (formally analyzed in Section~\ref{sec:complexity}), effectively managing both unique and repetitive $k$-mers. Additionally, we examine the effect of separating $k$-mer counts on compression ratios for cases involving non-uniform multiplicities.

\subsection{Paper Structure}
The remainder of this paper is organized as follows: Section~\ref{sec:preliminaries} provides essential background on de Bruijn graphs, Eulerian paths, and $k$-mer representations. Section~\ref{sec:algorithm} details the proposed \textbf{Minimum Compressed Text Representation (MCTR) algorithm}, including graph construction, our local Eulerization technique, Eulerian cover generation, and the final BWT compression stage. Section~\ref{sec:experiments} presents the experimental evaluation, assessing MCTR's performance on simulated and real data, and comparing its lossless compression characteristics against the lossy state-of-the-art tool \texttt{matchtigs}. Finally, Section~\ref{sec:conclusion} summarizes our findings and discusses future research directions.

\section{Preliminaries}\label{sec:preliminaries}
In this section, we introduce foundational concepts from graph theory and $k$-mer analysis that underpin the methods used in this paper.

\subsection{Sets and Multisets} \label{mset-rep}
A set is an unordered collection of distinct elements, where each element appears only once.

The number of elements in a set is called its cardinality or size, denoted by $|A|$. A subset of a set $A$ is a set $B$ such that all elements of $B$ are also in $A$.

A multiset allows elements to appear multiple times, with each occurrence referred to as its multiplicity. As with sets, the order of elements in a multiset is not important. In contrast to sets, elements in a multiset can have multiple occurrences. To represent a multiset $A$ in a compressed form, we list its elements along with their multiplicities, which is commonly known as the \textit{frequency representation}. In the union of multisets, the multiplicities of the same elements are summed.

To distinguish between sets and multisets, we use different notations: $\{ \}$ for sets and $\lbar \rbar$ for multisets. This distinction is important as it reflects whether or not repetitions are allowed in the collection.

Multisets can be represented in two common ways: a \textit{list representation} and a \textit{frequency representation}~\cite{mset}.

In the \textit{list representation}, the multiset is represented as a list or array that explicitly includes all repeated elements. Each occurrence of an element is stored individually without using a frequency count, which can lead to increased storage requirements. In the \textit{frequency representation}, a multiset is represented by storing the unique elements alongside their respective frequencies. The elements and their frequencies are stored in two separate lists: one for the elements and one for their corresponding counts.

For example, consider the multiset $\lbar AA, AT, GC, AA, AA, GC\rbar$. In the list representation, all occurrences of the elements are explicitly stored in the sequence. Alternatively, in the frequency representation, the multiset can be represented as $\{AA:3, AT:1, GC:2\}$. Equivalently, it can be represented as $\{AA, AT, GC\}$ with an associated frequency list \(\{3, 1, 2\}\), where each frequency corresponds to the count of the respective element in the original multiset. This example highlights the distinction between the two representation methods: the list representation explicitly enumerates every occurrence, while the frequency representation groups unique elements and annotates them with their frequencies.

\subsection{Basic Graph Definitions}
A directed graph $G = (V, E)$ consists of a vertex set $V = V(G)$ and an edge set $E = E(G)$. Each edge is an ordered pair of vertices $e = (u, v) \in E$, where $u, v \in V$, indicating a directed edge from $u$ to $v$. We also use notation $u \rightarrow v$ for such an edge.
The vertex $u$ is the tail of edge $e$, and $v$ is its head.
The out-degree of a vertex $v$ (denoted $d^+(v)$) is the number of edges leaving $v$, and the in-degree (denoted $d^-(v)$) is the number of edges entering $v$.
An imbalance of a vertex $v$ is defined as $\delta(v) = d^+(v) - d^-(v)$. If $d^+(v) = d^-(v)$ then $v$ is called balanced.

In an undirected graph, edges are unordered pairs $e = \{u, v\}$, meaning the direction between $u$ and $v$ is not specified. The underlying graph of a directed graph is the undirected graph formed by ignoring the edge directions.

\subsection{Trails and Eulerian Covers}
In a directed graph, a trail is a sequence of vertices and edges, written as $ T = v_0e_1v_1e_2 \ldots v_{\ell-1}e_{\ell}v_{\ell} $, where each edge $ e_i = (v_{i-1}, v_i) $, and all edges in the trail are distinct. To simplify notation, we often represent the trail as a sequence of vertices: $ v_0 \rightarrow v_1 \rightarrow \ldots \rightarrow v_{\ell} $. In undirected graphs, edges in the trail are undirected. If the starting vertex is the same as the ending vertex, the trail forms a circuit.

An Eulerian trail 
is a trail that visits every edge exactly once. If such a trail is a circuit, it is called an Eulerian circuit. A graph is Eulerian if it contains an Eulerian circuit. An Eulerian cover is a set of trails that collectively cover every edge of the graph exactly once.
A graph is local Eulerian if each its component is Eulerian. Clearly, in this case the minimum number of trails required to form an Eulerian cover is equal to the number of connected components. The process of adding edges to make a graph local Eulerian is called local Eulerization.

The Eulerian property~\cite{rahman2017basic} of a directed graph $G$ is formalized in the following theorem:

\begin{theorem}\label{thm:eulerian}
A directed graph $G = (V, E)$ is local Eulerian if and only if it is balanced, i.e., $ d^-(v) = d^+(v) $  for every vertex $v \in V$.
\end{theorem}

\subsection{$k$-mers and de Bruijn Graphs}
An alphabet $\Sigma$ is a finite set; any sequence of elements from $\Sigma$ is called a word. A word of length $k$ is known as a $k$-word or $k$-mer. For example, DNA sequences in biology can be seen as words over the alphabet $\Sigma = \{A, C, G, T\}$. Given a $k$-mer $w = w_1w_2\ldots w_k$, its $i$-prefix is 
$\pre_i(w) = w_1 \ldots w_i$, and its $i$-suffix is $\suf_i(w) = w_{k-i+1}\ldots w_k$. Specifically, the prefix and suffix of a $k$-mer of length $k$ refer to its $(k-1)$-prefix and $(k-1)$-suffix, respectively.

Given a set of $k$-mers $\M$, the corresponding de Bruijn graph $dBG(\M) = (V, E)$ is a directed graph where vertices are the $(k-1)$-prefixes and $(k-1)$-suffixes of the $k$-mers in $\M$. The edges represent the $k$-mers themselves, with each edge directed from the prefix to the suffix. Formally, $V = \cup_{w \in \M} \{\pre_{k-1}(w), \suf_{k-1}(w)\}$ and $E = \lbar (\pre_{k-1}(w), \suf_{k-1}(w)) : w \in \M \rbar$. When $\M$ is a multiset, $dBG(\M)$ becomes a multigraph, allowing multiple edges between the same pair of vertices.

In $dBG(\M)$, if a trail $T = u_1 \rightarrow u_2 \rightarrow \cdots \rightarrow u_t$ exists, where the suffix of $u_i$ matches the prefix of $u_{i+1}$, the readout sequence of the trail is defined as $\spell(T) = u_{1,1} \ldots u_{1,k-1} u_{2,k-1} \cdots u_{t,k-1}$, where $u_i=u_{i,1}u_{i,2}\ldots u_{i,k-1}$. The $k$-mer multiset $\puff(w)$ of a text sequence $w = w_1w_2\ldots w_{\ell}$, $(\ell \geq k)$ consists of all 
its patterns of length $k$, i.e., $w_iw_{i+1}\ldots w_{i+k-1}$ where $1 \leq i \leq \ell-k+1$. The $k$-mer multiset represented by a set of sequences $\W$ is defined as
$\puff(\W) = \cup_{w \in \W} \puff(w)$.

\subsection{Compressed Text Representation of $k$-mer Sets}
The {\em weight function} $\weight(\X)$ for the word set $\X$ is defined as the total length of $\X$, including separators between words, i.e., $\weight(\W) = \sum_{w \in \W} |w| + |\W| - 1$. For example, for $\W = \lbar \texttt{ATGC}, \texttt{ACGA}, \texttt{ACGTA} \rbar$, $\weight(\W) = 15$. A {\em compressed text representation} of a $k$-mer set $\M$ is a set of words $\W$ such that $\puff(\W) = \M$ and $\weight(\M)>\weight(\W)$. We define the {\em compression ratio} as:
\begin{align}
\cmpr_{\M}(\W) & = \frac{\weight(\M)}{\weight(\W)}. \label{eq:crm}
\end{align}
 The optimal compressed text representation minimizes $\weight(\W)$. The following theorem formalizes the relationship between the compression ratio and the optimal compressed representation:

\begin{theorem}\label{thm:MN}
If $\W$ is a compressed text representation of a $k$-mer set $\M$, then $\cmpr_{\M}(\W) < k + 1$ and the optimal compressed representation
consists of the minimum number of words, i.e.,
\begin{align}
\argmin_{\puff(\W) = \M} \weight(\W) = \argmin_{\puff(\W) = \M} |\W|. \label{eq:minW}
\end{align}
\end{theorem}
\begin{proof}
By definition, we have:
\begin{align}
\weight(\W) & = \sum_{w \in \W} |w| + |\W| - 1. \label{eq:wN}
\end{align}
Moreover,
\begin{align}
|\M| & = \sum_{w \in \W} ( |w| - k + 1 ) = \sum_{w \in \W} |w| - (k - 1)|\W|.   \label{eq:NM}
\end{align}
Combining equations \eqref{eq:NM} and \eqref{eq:wN}, we get:
\begin{align}
\weight(\W) & = |\M| + k|\W| - 1. \label{eq:wNM}
\end{align}
Therefore, $\weight(\W)$ decreases as $|\W|$ decreases, thus proving equation \eqref{eq:minW}.
Furthermore, $\weight(\W) \geq |\M| + k - 1$, and equality holds when $|\W| = 1$. Also, $\weight(\M) = (k+1)|\M| - 1$, so we have:
\begin{align}
\cmpr_{\M}(\W) \leq \frac{(k+1)|\M| - 1}{|\M| + k - 1} = k + 1 - \frac{k^2}{|\M| + k - 1} < k+1.
\end{align}
\end{proof}

\subsection{Burrows-Wheeler Transform and Run-Length Encoding}\label{sec:bwt}

The {\em Burrows-Wheeler Transform} ({\em BWT}), in combination with {\em Run-Length Encoding} ({\em RLE}), provides an effective two-step strategy for compressing repetitive sequences. While early implementations of the BWT required $O(n \log n)$ time and super-linear space, modern algorithms based on suffix arrays can construct the BWT in $O(n)$ time and space, making it highly practical for large-scale datasets.

More recent work has further lowered the theoretical cost of BWT construction. Kempa and Kociumaka~\cite{Kempa2022resolution} resolved the BWT conjecture, developing a string-synchronizing-set approach that achieves $O(n/\sqrt{\log n})$ time and requires only $O(n/\log n)$ workspace. In parallel, Bannai et al.~\cite{Bannai2021bijective} presented an $O(n)$-time algorithm for both the bijective BWT (BBWT) and the extended BWT (eBWT). Additionally, the efficiency of the subsequent RLE step has been studied by Bentley et al.~\cite{Bentley2020runs}, who analyzed the complexity of minimizing the number of runs in the BWT via optimal alphabet reordering.

The BWT is a reversible transformation that permutes the characters of a sequence $w$ to group identical characters together. The process involves creating a matrix $M(w)$ of all cyclic rotations of $w$, sorting the rows lexicographically, and extracting the last column, which is the $\text{BWT}(w)$. This resulting string is highly amenable to compression techniques like RLE, as illustrated below.

\begin{example}
\label{ex:bwt}
    Consider the string $w = \texttt{BANANA\$}$, where \texttt{\$} represents a unique end-of-sequence symbol. We assume the standard character order where \texttt{\$} $<$ \texttt{A} $<$ \texttt{B} $<$ \dots $<$ \texttt{Z}. Sorting the cyclic rotations of $w$ yields the following matrix:
    \[
        \begin{matrix}
            \texttt{\$BANANA} \\
            \texttt{A\$BANAN} \\
            \texttt{ANA\$BAN} \\
            \texttt{ANANA\$B} \\
            \texttt{BANANA\$} \\
            \texttt{NA\$BANA} \\
            \texttt{NANA\$BA}
        \end{matrix}
    \]
    
 \noindent
 The last column of this matrix, \texttt{ANNB\$AA}, is the BWT of the original string. Applying RLE to this result gives \texttt{AN2B\$A2}.
\end{example}
As the example illustrates, the key property of the BWT is its ability to cluster identical characters, which is then exploited by RLE. The transform is fully reversible, allowing the original sequence to be perfectly reconstructed from $\text{BWT}(w)$ and the primary index (the row number of the original string in the sorted matrix). The practical utility of this approach for genomics is rooted in the modern linear-time algorithms mentioned earlier, with further innovations providing memory-efficient variants for processing datasets that exceed available RAM~\cite{bwte}.
\subsection{Combining BWT and Minimum Compressed Text Representations}

In the context of genomic data compression, the BWT can be applied to the compressed text representations derived from the Eulerian covering of de Bruijn graphs. After obtaining a compressed representation $\W$ from a $k$-mer set $\M$, the BWT is applied to concatenation of all words in $\W$, separated by a special symbol. This process  enhances compressibility by clustering recurrent patterns, which enables the effective application of advanced compression techniques to the transformed data.

The combination of Eulerian de Bruijn graph traversal and the BWT facilitates a two-stage compression process: (1) representing the $k$-mer set compressedly through graph traversal, and (2) applying the BWT to exploit the inherent repetitive structures for additional compression. This approach ensures that the representation remains both lossless and highly compressible, making it particularly effective for handling extensive genomic datasets.

\section{Algorithm for Minimum Compressed Text Representation}\label{sec:algorithm}
    This section details our algorithm for generating a minimum compressed text representation of a $k$-mer set. The entire five-stage process is illustrated in Figure~\ref{fig:main}, which walks through the compression of the sample $k$-mer multiset $\M$ for $k=4$:
\begin{equation}
\M = \{\texttt{ATAC}, \texttt{ATCA}, \texttt{ATGA}, \texttt{ATGC}, \texttt{CATC}, \texttt{TCAT}, \texttt{TGCT}\}
\end{equation}

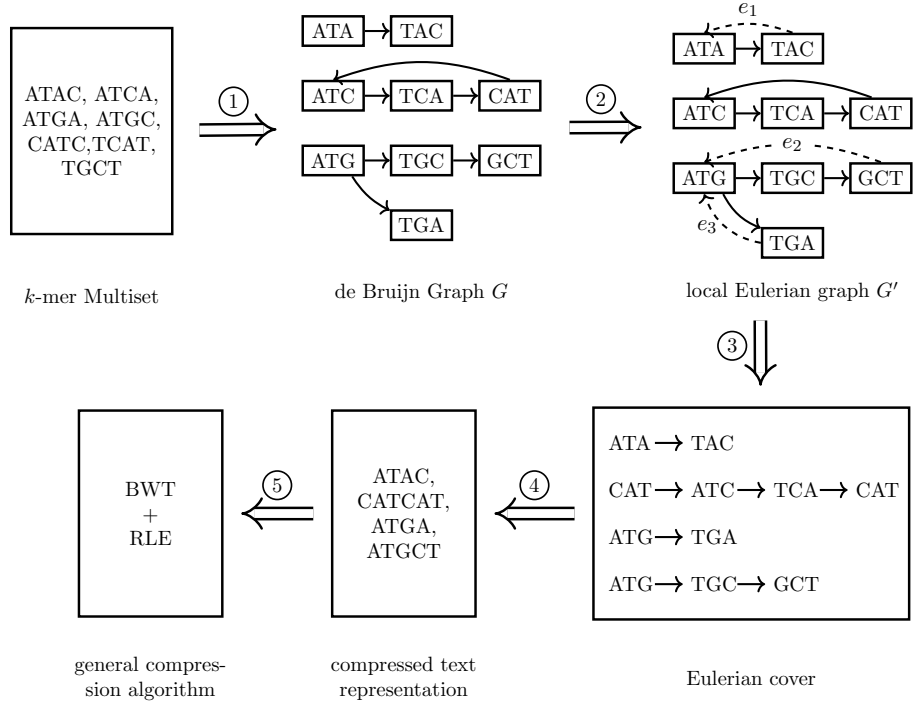
\begin{figure}[h]
	\centering
	\resizebox{0.8\textwidth}{!}
	{%
		\begin{tikzpicture}[node distance=1.5cm,on grid]
			\definecolor{black}{RGB}{0,0,0}

			\tikzset{double arrow/.style={- implies,
						double distance=1.5mm,line width=1.2pt,draw=black}}

			\tikzstyle{Set} = [rectangle,line width=1.2pt,minimum width=2cm,minimum height=3.5cm,text centered,text width =2.5cm,draw=black]

			\tikzstyle{kmer} = [rectangle, line width=1.2pt,minimum width=0.8cm,minimum height=0.5cm,text centered,text width =0.8cm,draw=black]

			\tikzstyle{seq} = [rectangle, line width=1.2pt,minimum width=0.6cm,minimum height=0.5cm,text centered,text width =0.65cm,draw=white,draw opacity=0]

			\tikzstyle{background} = [rectangle,line width=1.2pt,minimum width=4.3cm, minimum height=3.6cm,draw=black]

			\tikzstyle{arrow} = [double,double distance=2pt,
			line width=1pt,-Latex]
			\node [Set] (1) {ATAC, ATCA, ATGA, ATGC, CATC,TCAT, TGCT};

			\node [kmer] (2a)  [right=of 1,xshift=2.6cm,yshift=1.7cm] {ATA};
			\node [kmer] (2b)  [right=of 2a] {TAC};
			\node [kmer] (2c)  [below=of 2a,yshift=0.4cm] {ATC};
			\node [kmer] (2d)  [right=of 2c] {TCA};
			\node [kmer] (2e)  [right=of 2d] {CAT};
			\node [kmer] (2f)  [below=of 2c,yshift=0.4cm] {ATG};
			\node [kmer] (2g)  [right=of 2f] {TGC};
			\node [kmer] (2h)  [right=of 2g] {GCT};
			\node [kmer] (2i)  [below=of 2g,yshift=0.4cm] {TGA};

			\draw[->,line width=1pt] (2a) -- (2b);
			\draw[->,line width=1pt] (2c) -- (2d);
			\draw[->,line width=1pt] (2d) -- (2e);
			\draw[->,line width=1pt] (2f) -- (2g);
			\draw[->,line width=1pt] (2g) -- (2h);
			\draw[->,line width=1pt,color=black] (2e.{+90}) to[bend right=20]  (2c.{+90});
			\draw[->,line width=1pt] (2f.{+320}) to[bend right=15] (2i.{+155});

			\begin{scope}[xshift=5cm]
				\node [kmer] (3a)  [right=of 2a,xshift=4.8cm,,yshift=-0.3cm] {ATA};
				\node [kmer] (3b)  [right=of 3a] {TAC};
				\node [kmer] (3c)  [below=of 3a,yshift=0.4cm] {ATC};
				\node [kmer] (3d)  [right=of 3c] {TCA};
				\node [kmer] (3e)  [right=of 3d] {CAT};
				\node [kmer] (3f)  [below=of 3c,yshift=0.4cm] {ATG};
				\node [kmer] (3g)  [right=of 3f] {TGC};
				\node [kmer] (3h)  [right=of 3g] {GCT};
				\node [kmer] (3i)  [below=of 3g,yshift=0.4cm] {TGA};

				\draw[->,line width=1pt] (3a) -- (3b);
				\draw[->,line width=1pt] (3c) -- (3d);
				\draw[->,line width=1pt] (3d) -- (3e);
				\draw[->,line width=1pt] (3f) -- (3g);
				\draw[->,line width=1pt] (3g) -- (3h);

				\draw[->,dashed,line width=1pt,color=black] (3b.{+90}) to [bend right=20]  node (e_{1}) [midway, above] {\large $e_{1}$}  (3a.{+90});

				\draw[->,line width=1pt] (3f.{+320}) to[bend right=20] (3i.{+155});

				\draw[->,line width=1pt,color=black] (3e.{+120}) to[bend right=20]  (3c.{+90});

				\draw[->,dashed,line width=1pt] (3h.{+90}) to [bend right=20] node[fill=white,midway,sloped] {\large $e_{2}$} (3f.{+90});

				\draw[->,dashed,line width=1pt,color=black] (3i.{+180}) to[bend left=30] node[midway, left] {\large $e_{3}$} (3f.{+270});
			\end{scope}

			\begin{scope}[yshift=-3cm]
				\node [seq] (4a)  [below=of 3f,xshift=-1.3cm,yshift=-3cm] {ATA};
				\node [seq] (4b)  [right=of 4a,xshift=-0.1cm] {TAC};
				\node [seq] (4c)  [below=of 4a,yshift=0.7cm] {CAT};
				\node [seq] (4d)  [right=of 4c,xshift=-0.1cm] {ATC};
				\node [seq] (4f)  [right=of 4d,xshift=-0.1cm] {TCA};
				\node [seq] (4z)  [right=of 4f,xshift=-0.1cm] {CAT};
				\node [seq] (4g)  [below=of 4c,yshift=0.7cm] {ATG};
				\node [seq] (4h)  [right=of 4g,xshift=-0.1cm] {TGA};
				\node [seq] (4i)  [below=of 4g,yshift=0.7cm] {ATG};
				\node [seq] (4j)  [right=of 4i,xshift=-0.1cm] {TGC};
				\node [seq] (4k)  [right=of 4j,xshift=-0.1cm] {GCT};

				\draw[->,line width=1pt] (4a) to (4b) ;
				\draw[->,line width=1pt] (4c) to (4d) ;
				\draw[->,line width=1pt] (4d) to (4f) ;
				\draw[->,line width=1pt] (4f) to (4z) ;
				\draw[->,line width=1pt] (4g) to (4h) ;
				\draw[->,line width=1pt] (4i) to (4j) ;
				\draw[->,line width=1pt] (4j) to (4k) ;

			\end{scope}

			\begin{scope}[xshift=-4cm]
				\node  [rectangle, line width=1.2pt,minimum width=2cm,minimum height=3.5cm,text centered,text width =2.2cm,draw=black] (5) [left=of 4a,xshift=-2.3cm,yshift=-1.2cm,] {ATAC, \ CATCAT, \ ATGA,\ ATGCT };
			\end{scope}

			\begin{scope}[xshift=-4cm]
				\node  [rectangle, line width=1.2pt,minimum width=2cm,minimum height=3.5cm,text centered,text width =2.2cm,draw=black] (6) [left=of 5,xshift=-2.8cm] {BWT \\+ \\ RLE };
			\end{scope}

			\begin{scope}[on background layer]
				\node (2) [fit = (2a)(2b)(2f)(2h)(2i)] {};
				\node (3) [fit = (e_{1})(3f)(3h)(3i)] {};
				\node[background] (4) [fit = (4a)(4b)(4z)(4k)] {};

			\end{scope}

			\node[below of=1,node distance=2.8cm] {$k$-mer Multiset};
			\node[below of=2,node distance=2.8cm] {de Bruijn Graph $G$};
			\node (A) [below of=3,node distance=2.8cm] {local Eulerian graph $G'$};
			\node[below of=4,node distance=2.8cm] {Eulerian cover};
			\node[below of=5,node distance=2.8cm,rectangle,minimum width=3cm,text centered,text width =3cm] {compressed text representation};
			\node[below of=6,node distance=2.8cm,rectangle,minimum width=3cm,text centered,text width =3cm] {general compression algorithm};

			\draw[double arrow,shorten >=3mm, shorten <=4.4mm] (1) --  node[midway, above=1mm, sloped] {\tikz[baseline=(X.base)]{\node[draw,thick,circle,inner sep=2pt] (X) {1};}} (2);

			\draw[double arrow ,shorten >=4mm, shorten <=3.5mm] (2) --  node[midway, above=1mm, sloped,xshift=-0.1cm] {\tikz[baseline=(X.base)]{\node[draw,thick,circle,inner sep=2pt] (X) {2};}} (3);

			\draw[double arrow,shorten >=3.5mm, shorten <=2mm] (A.{+208}) --  node[midway, left=1mm,yshift=0.2cm] {\tikz[baseline=(X.base)]{\node[draw,thick,circle,inner sep=2pt] (X) {3};}} (4.{+85});

			\draw[double arrow, shorten >=3mm,shorten <=3mm] (4) --  node[midway, above=1mm, sloped] {\tikz[baseline=(X.base)]{\node[draw,thick,circle,inner sep=2pt] (X) {4};}} (5);

			\draw[double arrow,shorten >=3mm, shorten <=3mm] (5) --  node[midway, above=1mm, sloped] {\tikz[baseline=(X.base)]{\node[draw,thick,circle,inner sep=2pt] (X) {5};}} (6);

		\end{tikzpicture}}
	\caption{Flowchart of the Minimum Compressed Text Representation Algorithm}
    \label{fig:main}
\end{figure}

The stages of the algorithm are as follows:

\begin{enumerate}[label=\protect\circled{\arabic*}, leftmargin=2.5em, labelsep=0.3em]
    \item \textbf{Constructing the de Bruijn Graph}: Each $k$-mer in $\M$ corresponds to a directed edge from its $(k-1)$-prefix to its $(k-1)$-suffix, forming the de Bruijn graph $G$.

    \item \textbf{Local Eulerization}: To ensure all edges can be covered by a minimum number of paths, we transform $G$ into a local Eulerian graph $G'$ by strategically adding a minimum set of new edges, $E'$.

    \item \textbf{Generating the Eulerian Cover}: We use Hierholzer's algorithm to compute an Eulerian cover $\T$ of $G'$. This process finds a minimum set of trails that collectively traverse every original edge in $G$ exactly once.

    \item \textbf{Obtaining the Compressed Text Representation}: Each trail in the cover $\T$ is spelled out into a text string. This produces the minimum compressed text representation $\W = \{ \texttt{ATAC}, \texttt{CATCAT}, \texttt{ATGA},\texttt{ATGCT} \}$. This representation is significantly more compact, with a weight reduction from $\weight(\M)=34$ to $\weight(\W)=22$.

    \item \textbf{Applying BWT and RLE}: For a final layer of compression, the strings in $\W$ are concatenated into a single sequence $S = \texttt{ATAC,CATCAT,ATGA,ATGCT\$}$ (using `\texttt{,}' as a delimiter and `\texttt{\$}' as a terminator). The BWT is applied to $S$, yielding $S_{BWT} = \texttt{TTACGTC\$C,,AT,GTTCAAAAA}$. This new string, with its clustered characters, is then efficiently compressed by RLE to produce $S_{RLE} = \texttt{T2ACGTC\$C,2AT,GT2CA5}$, which has a final weight of 19.

\end{enumerate}

This five-stage pipeline provides a formally grounded method for $k$-mer set compression, seamlessly integrating graph-theoretic optimization with well-established lossless compression techniques.

\subsection{Algorithm Overview}
We now present the detailed algorithm, named the \textbf{Minimum Compressed Text Representation (MCTR) Algorithm}, which formalizes the steps outlined above.

\begin{algorithm}[H]
    \caption{Minimum Compressed Text Representation (MCTR) Algorithm}
    \label{alg:MCTR}
    \begin{algorithmic}[1]
        \Require A $k$-mer set $\M$
        \Ensure A minimum compressed text representation $\W$
        \State Construct the de Bruijn graph $G = dBG(\M)$ from $\M$
        \State Apply local Eulerization (Algorithm~\ref{alg:local_Eulerization}) to obtain a local Eulerian graph $G'$
        \State Generate an Eulerian cover $\T$ of $G'$ using Hierholzer's algorithm (Algorithm~\ref{alg:Hierholzer})
        \State Spell out the trails in $\T$ to form the compressed text representation $\W$
        \State Compress $\W$ using BWT and RLE
    \end{algorithmic}
\end{algorithm}

The MCTR algorithm is designed for high efficiency, capable of processing large-scale $k$-mer datasets. A detailed, step-by-step analysis of its linear time and space complexity, along with a discussion of its theoretical advantages compared to related methods, is provided in Section~\ref{sec:complexity}.

\subsection{Local Eulerization}
In Step~\circled{2}, the \textbf{local Eulerization Algorithm} (Algorithm~\ref{alg:local_Eulerization}) adds edges to the de Bruijn graph $G$ to produce a local Eulerian graph $G'$. Lemma~\ref{thm:local_Euler} demonstrates that this edge addition method is optimal, minimizing the number of edges required.

\begin{algorithm}[H]
    \caption{Local Eulerization Algorithm}
    \label{alg:local_Eulerization}
    \begin{algorithmic}[1]
        \Require A de Bruijn graph $G = (V, E)$
        \Ensure A local Eulerian graph $G' = (V, E \cup E')$, added edge set $E'$
        \State Initialize an empty set $E'$ and two sequences $S^+$ and $S^-$
        \For{each vertex $v \in V$}
            \State Compute the in-degree $d^-(v)$ and out-degree $d^+(v)$ of $v$
            \State $\delta(v) \gets d^-(v) - d^+(v)$
            \If{$\delta(v) > 0$}
                \State Append $\delta(v)$ copies of $v$ to $S^+$
            \ElsIf{$\delta(v) < 0$}
                \State Append $-\delta(v)$ copies of $v$ to $S^-$
            \EndIf
        \EndFor
        \For{each pair $(s_i^+, s_i^-)$ from $S^+$ and $S^-$}
            \State Add the edge $(s_i^-, s_i^+)$ to $E'$
        \EndFor
        \State \Return $G' = (V, E \cup E')$ and $E'$
    \end{algorithmic}
\end{algorithm}

\begin{lemma}\label{thm:local_Euler}
The output graph $G' = (V, E \cup E')$ from Algorithm~\ref{alg:local_Eulerization} is a local Eulerian graph. Furthermore, the algorithm adds the minimum number of edges required to achieve this.
\end{lemma}

\begin{figure}[h]
	\centering
	\resizebox{0.8\textwidth}{!}
	{
		\begin{tikzpicture}
			[node distance=2cm,on grid,>={Stealth[round]},
			every place/.style={minimum size=4mm,thick}]

			\node [place] (8) [label={right:{\large 9}}] {};
			\node [place] (9) [left=of 8,label={left:{\large 6}}] {};
			\node [place] (w1) [left=of 9,label={right:{\large 4}}] {};
			\node [place] (3) [left=of w1,label={left:{\large 5}}] {};
			\node [place] (2) [left=of 3,label={left:{\large 2}}] {};
			\node [place] (1) [above=of 2,label={left:{\large 1}}] {};

			\node [place] (6) [above=of 9,label={left:{\large 7}}] {};
			\node [place] (7) [right=of 6,label={right:{\large 8}}] {};
			\node [place] (4) [right=of 1,xshift=1cm,label={above:{\large 3}}] {};

			\draw[-Latex,line width=1pt] (2) -- (1);
			\draw[-Latex,line width=1pt] (3) -- (4);
			\draw[-Latex,line width=1pt] (4) -- (w1);
			\draw[-Latex,line width=1pt] (w1) -- (3);
			\draw[-Latex,line width=1pt] (9) -- (6);
			\draw[-Latex,line width=1pt] (6) -- (7);
			\draw[-Latex,line width=1pt] (8) -- (9);

			\begin{scope}[xshift=13cm]
				\node [place] (c8) [label={right:{\large 9}}] {};
				\node [place] (c9) [left=of c8,label={150:{\large 6}}] {};
				\node [place] (cw1) [left=of c9,label={60:{\large 4}}] {};
				\node [place] (c3) [left=of cw1,label={left:{\large 5}}] {};
				\node [place] (c2) [left=of c3,label={left:{\large 2}}] {};
				\node [place] (c1) [above=of c2,label={left:{\large 1}}] {};

				\node [place] (c6) [above=of c9,label={left:{\large 7}}] {};
				\node [place] (c7) [right=of c6,label={right:{\large 8}}] {};
				\node [place] (c4) [right=of c1,xshift=1cm,label={above:{\large 3}}] {};

				\draw[-Latex,line width=1pt] (c2) -- (c1);
				\draw[-Latex,line width=1pt] (c3) -- (c4);
				\draw[-Latex,line width=1pt] (c4) -- (cw1);
				\draw[-Latex,line width=1pt] (c9) -- (c6);
				\draw[-Latex,line width=1pt] (c9) -- (c6);
				\draw[-Latex,line width=1pt] (c6) -- (c7);
				\draw[-Latex,line width=1pt] (c9) -- (c8);
				\draw[-Latex,line width=1pt] (cw1) -- (c3);
				\draw (c8.{135}) edge[-Latex,dashed,line width=1.2pt,bend right=20] (c9.{45});
				\draw[-Latex,dashed,line width=1.2pt] (c7) -- (c9);
				\draw (c1) edge[-Latex,dashed,line width=1.2pt,bend left=20] (c2);

			\end{scope}

			\begin{scope}[on background layer]
				\node (r1) [fit=(1)(2)(7)(8)] {};
				\node (r2) [fit=(c1)(c2)(c7)(c8)] {};
			\end{scope}
			\draw[- implies, double distance=1.5mm,line width=1.2pt,draw=black,shorten >=3mm, shorten <=5mm] (r1) --  node[midway, above, sloped]{\large local Eulerizing} (r2);

			\node[below of=r1, node distance=2.5cm] {\large $G$};
			\node[below of=r2, node distance=2.5cm] {\large $G^{'}$};

		\end{tikzpicture}}
		\caption{Illustration of Local Eulerization on Graph $G$ to Obtain Graph $G'$}
		\label{fig:eg_local_eulerian}
	\end{figure}
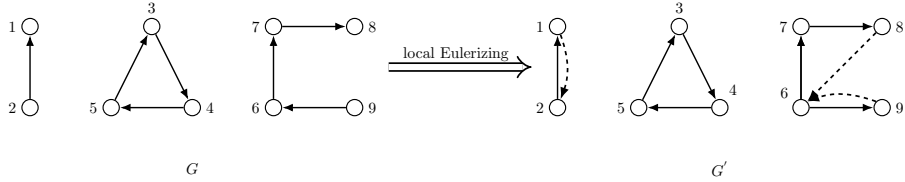

\begin{proof}
First, we show that $G'$ is local Eulerian. In the original graph $G$, each unbalanced vertex $v$ has an imbalance $\delta(v) = d^-(v) - d^+(v) \neq 0$. When $\delta(v) > 0$ (or $\delta(v) < 0$), the added edges in $E'$ correct this imbalance by adding edges with $v$ as a tail (resp., as a head). After these additions, all vertices in $G'$ are balanced. By Theorem~\ref{thm:eulerian}, $G'$ is thus local Eulerian.

To establish the minimality of $E'$, note that each added edge in $E'$ increases the in-degree of one vertex and the out-degree of another by one. To balance all unbalanced vertices in $G$, the total imbalance must be zero, and the number of added edges is exactly half the sum of the absolute values of imbalances:
\begin{align}
|E'| = \frac{1}{2} \sum_{v \in V} |\delta(v)|.
\end{align}
This is the minimum number of edges needed to balance the graph, demonstrating that the algorithm adds the fewest edges necessary for local Eulerization.
\end{proof}

\subsection{Generating the Compressed Text Representation}
\label{sec:gen_text}

In Step~\circled{3}, we generate the Eulerian cover, and in Step~\circled{4}, we spell out the corresponding strings to form the final compressed text representation $\mathbb{W}$. A significant challenge arises here: the Eulerized graph $G'$ contains artificial edges from the set $E'$ that do not correspond to any $k$-mers in the original input $\mathbb{M}$. A naive spelling of the complete Eulerian tour found in $G'$ would result in a single, incorrect sequence containing artificial $k$-mers, thereby corrupting the data.

To solve this, we introduce the \textbf{Eulerian Cover Generation Algorithm (Algorithm 3)}, a modified version of Hierholzer's algorithm. The key innovation is to track the origin of each traversed edge. The algorithm uses these artificial edges from $E'$ solely as non-emitting ``bridges'' to connect valid trails from the original graph $G$.

The process works as follows: the algorithm traverses the graph, building up a path. When it backtracks (i.e., pops a vertex from its stack), it reconstructs the corresponding sequence in reverse. Crucially, if the algorithm backtracks over an artificial edge, it recognizes this as the end of a valid, contiguous trail from the original graph. At this point, it finalizes the currently constructed string, adds it to the output set $\mathbb{W}$, and resets to begin a new string for the next segment. This mechanism ensures that the artificial edges are never spelled out, guaranteeing that the final representation $\mathbb{W}$ is a lossless and correct representation of the original $k$-mer multiset $\mathbb{M}$.

\begin{algorithm}[H]
\caption{Eulerian Cover Generation Algorithm}
\label{alg:Hierholzer}
\begin{algorithmic}[1]
\State \textbf{Input:} A local Eulerian graph $G'=(V, E \cup E')$, with added edges $E'$
\State \textbf{Output:} A minimum compressed text representation $\mathbb{W}$
\Statex
\State Initialize an empty list $\mathbb{W}$
\State Initialize all edges in $E \cup E'$ as unvisited
\While{there exists an unvisited edge in $E \cup E'$}
    \State Choose an arbitrary vertex $u$ with an unvisited outgoing edge as the start of a new tour
    \State Initialize an empty path $P$ and an empty stack $S$
    \State Push $u$ onto $S$
    \While{$S$ is not empty}
        \State Let $v$ be the vertex at the top of $S$
        \If{$v$ has an unvisited outgoing edge $(v, w)$}
            \State Mark $(v, w)$ as visited
            \State Push $w$ onto $S$
        \Else
            \State Pop $v$ from $S$
            \State Prepend $v$ to path $P$
        \EndIf
    \EndWhile
    \State \Comment{The complete tour is now in $P$. Now, split it by artificial edges.}
    \State Initialize an empty string $s$
    \For{$i \leftarrow 1$ to $|P|-1$}
        \State Let $v_i = P[i]$ and $v_{i+1} = P[i+1]$
        \If{$s$ is empty}
            \State $s \leftarrow \text{label}(v_i)$ \Comment{Start string with the full (k-1)-mer label}
        \EndIf

        \If{the edge $(v_i, v_{i+1}) \in E'$}
            \State \Comment{Found an artificial edge; end the current string.}
            \State Append $s$ to $\mathbb{W}$
            \State $s \leftarrow \text{empty string}$
        \Else
            \State \Comment{Real edge; append the last character of the next vertex.}
            \State $s \leftarrow s + \text{last\_char}(\text{label}(v_{i+1}))$
        \EndIf
    \EndFor
    \If{$s$ is not empty} \Comment{Append the final string if it wasn't ended by an artificial edge.}
        \State Append $s$ to $\mathbb{W}$
    \EndIf
\EndWhile
\State \textbf{return} $\mathbb{W}$
\end{algorithmic}
\end{algorithm}

\subsection{Applying BWT and RLE}
In Step~\circled{5}, the compressed text representation is further compressed using the BWT followed by the RLE. The sequences in $\W$ are concatenated into a single string, with each component separated by a unique delimiter (e.g., `,') and terminated by an end-of-sequence marker (e.g., `\$'). This concatenated sequence is then transformed using the BWT, which rearranges characters to cluster similar elements, enhancing the effectiveness of the RLE.

%
Combining the BWT and the RLE ensures that the final compressed representation of the $k$-mer set is efficient for storage. The overall time and space complexity for this step are both $O(n)$, where $n$ is the length of the concatenated sequence, making this approach suitable for large-scale genomic datasets.

\subsection{Comparison of List and Frequency Representations for $k$-mer Datasets}
As discussed in Subsection \ref{mset-rep}, the input to the first step of the algorithm is provided in the list representation of the $k$-mer dataset. Alternatively, if the frequency representation is used, the input in Step \circled{1} consists of the unique elements of the $k$-mer dataset, while the corresponding frequency sequence is produced alongside the processed data in Step \circled{4}. The effectiveness of these two approaches depends on the frequency distribution of elements within the $k$-mer dataset.

The storage requirements for these representations differ significantly based on the frequency of elements. For instance, when the frequency of a $k$-mer $w$, denoted as $c(w)$, satisfies \(c(w) \leq 9\), we have
\begin{itemize}
    \item[i)] in the frequency representation, two additional characters are required: one as a separator and one digit $(1, 2, \ldots, 9)$ to represent the frequency.
    \item[ii)] in the list representation, \(c(w)\) extra characters are needed to store the last character of the $k$-mer, with the worst-case requirement of \((k+1) \times c(w)\) characters.
\end{itemize}

The choice between list representation and frequency representation is context-dependent. The frequency representation is generally more storage-efficient for datasets with high repetition of elements, whereas the list representation is simpler to implement and may be more suitable for datasets with lower repetition rates.

\subsection{Complexity Analysis and Comparison with Matchtigs}
\label{sec:complexity}

Here, we provide a formal analysis of the time and space complexity of our MCTR algorithm and clarify its relationship with other state-of-the-art methods like Matchtigs~\cite{Schmidt2023matchtigs}, addressing the important points raised during the review process.

\subsubsection{Complexity of the MCTR Algorithm}
Our algorithm is designed to operate in linear time and space with respect to the input size. Let $|\mathbb{M}|$ be the total number of $k$-mers in the input multiset, $|V|$ be the number of unique $(k-1)$-mers (vertices), and $|E| = |\mathbb{M}|$ be the number of edges in the de Bruijn graph. The complexity of each step is as follows:

\begin{description}
    \item[Step~\circled{1}: Construct de Bruijn graph] Using a hash table to store and retrieve vertices (the $(k-1)$-mers), the graph can be constructed in $O(|\mathbb{M}| \cdot k)$ time. Since $k$ is a small, fixed constant, the complexity is effectively linear, i.e., $O(|\mathbb{M}|)$. The space required is $O(|V| + |E|)$ to store the graph structure.

    \item[Step~\circled{2}: Local Eulerization] This step (Algorithm 2) involves a single pass over all vertices to compute degrees and imbalances, which takes $O(|V|)$ time. Identifying vertices for the sets $S^+$ and $S^-$ and subsequently adding the $|E'|$ new edges also requires time proportional to $|V|$ and $|E'|$. Since $|E'| \le |E|$, this entire process is bounded by $O(|V|+|E|)$, which is linear.

    \item[Step~\circled{3}: Generate Eulerian cover] We use Hierholzer's algorithm (Algorithm 3), which is known to find an Eulerian path/circuit in time proportional to the number of edges and vertices in the graph. For the Eulerized graph $G'=(V, E \cup E')$, this step takes $O(|V| + |E \cup E'|) = O(|V|+|E|)$ time.

    \item[Step~\circled{4}: Spell out trails] This process involves traversing the generated paths and concatenating characters from the vertex labels. The total work is proportional to the total length of the output strings, which we denote as $N$. This is an $O(N)$ operation, and since $N$ is on the order of $|\mathbb{M}|$, it is also linear.

    \item[Step~\circled{5}: BWT and RLE Compression] Modern implementations of the Burrows-Wheeler Transform, based on suffix arrays, have a time and space complexity of $O(N)$, where $N$ is the length of the concatenated input string. The subsequent Run-Length Encoding is a single pass, also taking $O(N)$ time.
\end{description}

Since each step of the MCTR algorithm has a linear time and space complexity, the overall algorithm is linear with respect to the input size.

\subsubsection{Comparison with Matchtigs}
A key question is how our algorithm achieves linear time while the optimal algorithm for a related problem, addressed by Matchtigs~\cite{Schmidt2023matchtigs}, requires $O(n^3m)$ time. The distinction lies in the underlying graph representation and the precise problem being solved.

\begin{itemize}
    \item \textbf{Matchtigs} operates on an \textit{overlap graph}, where vertices can represent entire sequencing reads and edges denote overlaps. Finding a minimal representation in such a graph is related to the NP-hard Shortest Common Superstring problem. Matchtigs provides a near-optimal solution by solving a minimum weight perfect matching problem, which is computationally intensive.

    \item \textbf{Our MCTR algorithm}, in contrast, operates on a \textit{de Bruijn graph}, where vertices are $(k-1)$-mers and edges are $k$-mers. This is a more constrained and structured representation. Our goal is to find a minimum \textit{edge path cover} for this specific graph. This problem is known to be solvable efficiently. By first making the graph Eulerian (our Local Eulerization step), the problem is reduced to finding a set of paths that cover all original edges, which can be accomplished in linear time using a modified Hierholzer's algorithm.
\end{itemize}

In summary, the linear complexity of our algorithm is a direct result of leveraging the structural properties of the de Bruijn graph and formulating the compression problem as an optimal Eulerian cover problem. This formulation is fundamentally different from the more general and computationally complex problem addressed by Matchtigs, thus explaining the difference in time complexity.

\subsubsection{Comparison with \texttt{matchtigs}}
A key aspect distinguishing our MCTR algorithm is its guaranteed linear time complexity, contrasting sharply with related approaches like `matchtigs' \cite{Schmidt2023matchtigs}, which involves solving a minimum weight perfect matching problem, a computationally more intensive task. This difference stems primarily from the specific problem formulation and the balancing strategy employed within the de Bruijn graph framework.

Both MCTR and ``matchtigs'' operate on variations of the de Bruijn graph constructed from $k$-mers (or unitigs derived from them). The core challenge in generating a concise representation lies in converting the typically non-Eulerian graph into one or more Eulerian paths or circuits.

\begin{itemize}
\item \textbf{\texttt{matchtigs}} aims to find a minimum-length plain text representation of the $k$-mer set by allowing repeated $k$-mers, formulated as a minimum-weight path cover problem in the de Bruijn graph.
To achieve this, it addresses the non-Eulerian nature of the graph by finding a minimum weight perfect matching between nodes with unbalanced in-degrees and out-degrees. The ``weight" typically corresponds to the shortest path distance between unbalanced nodes within the graph. This requires computationally expensive steps: calculating numerous shortest paths (e.g., using Dijkstra's algorithm) and then solving the minimum-weight perfect matching problem (e.g., using Blossom V) has a worst-case complexity of $O(N^3)$, where $N$ is the number of unbalanced nodes, and requires precomputing shortest paths between them.
While yielding high-quality paths, this pursuit of minimum weight balancing incurs significant computational cost.
Furthermore, while ``matchtigs" preserves the set of $k$-mers (i.e., it is lossless in terms of $k$-mer presence), it does not retain $k$-mer abundance information.

\item \textbf{Our MCTR algorithm}, conversely, prioritizes lossless representation and computational efficiency by targeting the theoretically minimal number of paths required to cover all edges. Our Local Eulerization step (Algorithm~\ref{alg:local_Eulerization}) achieves graph balancing by adding the minimum necessary number of edges, pairing unbalanced nodes in an arbitrary but efficient manner (linear time scan), without calculating shortest paths or minimizing total weight. This simplification reduces the balancing step to linear time. Subsequently, finding the Eulerian cover using Hierholzer's algorithm (Algorithm~\ref{alg:Hierholzer}) also operates in linear time with respect to the number of edges.
\end{itemize}

In summary, MCTR achieves linear time complexity ($O(|\M|)$, where $|\M|$ is the number of input $k$-mers) by formulating the problem as finding a minimum cardinality edge path cover on the de Bruijn graph and employing a computationally inexpensive balancing strategy. This contrasts with `matchtigs', which solves a more complex minimum weight balancing problem to optimize path quality, leading to higher computational complexity but resulting in a lossy representation. The efficiency of MCTR stems directly from leveraging the de Bruijn graph's structure and adopting a problem formulation focused on lossless representation with minimal path count rather than minimal path weight.

\section{Numerical Experiments}\label{sec:experiments}

We conducted experiments on both simulated and real datasets to evaluate the performance of our Minimum Compressed Text Representation (MCTR) algorithm. Our implementation was developed in Rust, leveraging its performance and safety features for handling graph operations efficiently. All experiments were performed on a Linux server with an Intel Xeon CPU and 64 GB RAM.  We evaluated performance based on key metrics: \textbf{core algorithm execution time} (excluding optional verification for MCTR), the \textbf{raw (uncompressed) output file size} ($\weight(\W)$), and the resulting \textbf{compression ratio} as defined in Eq.~(\ref{eq:crm}), calculated versus the size of the intermediate \texttt{.reads} file ($\weight(\M)$).

\subsection{Performance on Simulated Data}
First, we demonstrate the inherent characteristics of the MCTR algorithm in a controlled setting, relating its performance to theoretical bounds. According to Theorem~\ref{thm:MN}, the upper bound for the compression ratio of the intermediate text representation $\W$ (before secondary compression like BWT or gzip) is $k+1$. Higher $k$-mer set density leads to higher compression ratios.

In simulations using $k=10$ and varying the sampling ratio $r$ (proportion of $\Sigma^k$ sampled), we confirmed this relationship. As illustrated in Figure~\ref{fig:sim10}, the compression ratio $\cmpr_{\M}(\W)$ of the MCTR text representation increases with $r$, approaching the theoretical upper bound $k+1$. Higher $k$-mer multiplicities also led to improved compression ratios (data consistent with Figure~\ref{fig:sim10}), aligning with theoretical expectations. This validates the fundamental compression principle of the MCTR approach.

\begin{figure}[h]
    \centering
    \includegraphics[width=0.8\textwidth]{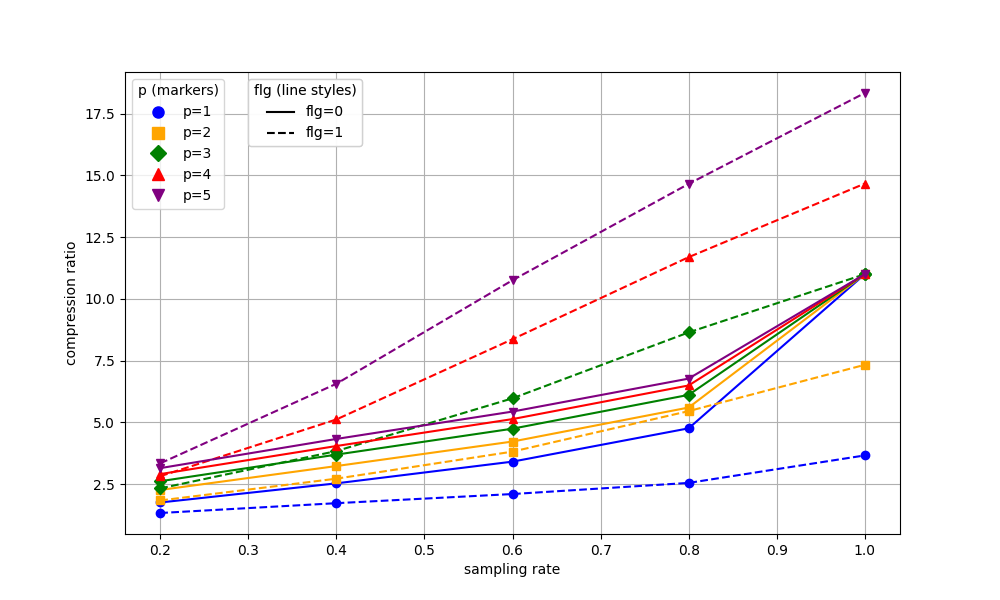}
    \caption{Simulation results with $k=10$: Compression ratio ($\cmpr = \weight(\M) / \weight(\W)$) of the MCTR compressed text representation $\W$ as a function of the sampling rate $r$. Consistent with Theorem~\ref{thm:MN}, the ratio approaches $k+1$ as the $k$-mer space becomes denser.}
    \label{fig:sim10}
\end{figure}

\subsection{Benchmark Performance and Comparison with State-of-the-Art}
Regarding benchmarks on standard datasets and comparisons with state-of-the-art tools, we evaluated MCTR against \texttt{greedytigs} (from the \texttt{matchtigs} toolkit v2.1.9 \cite{Schmidt2023matchtigs}) using the \textit{Escherichia coli} K-12 MG1555 reference genome (NCBI: GCF\_000005845.2) as a basis. We created two benchmark datasets:

\begin{enumerate}
    \item \textbf{Simulated Dataset (\texttt{sim}):} Low coverage (1x) Illumina-like paired-end reads generated using \texttt{wgsim} with a 1\% error rate (\texttt{sim\_11Mbp\_1x.fq}).
    \item \textbf{Real Dataset (\texttt{real}):} A subset of 20,000 real Illumina reads (SRA: SRR8185300 subsampled via ``\texttt{seqtk sample}; \texttt{real\_ecoli\_subset\_20k.fq}'').
\end{enumerate}

We compared the following methods for $k \in \{31, 51, 71\}$:

\begin{itemize}
    \item \textbf{\texttt{MCTR (list)}:} Our algorithm using the \textit{list representation}, providing a \textbf{lossless} representation of the input $k$-mer multiset. $\weight(\W)$ is the size of the \texttt{.ctr} file.
    \item \textbf{\texttt{MCTR (frequency)}:} Our algorithm using the \textit{frequency representation}, providing a \textbf{lossless} representation via \texttt{.ctr} and \texttt{.cnt} files. $\weight(\W)$ is the sum of the sizes of both files.
    \item \textbf{\texttt{GreedyTigs (ctr)}:} The \texttt{greedytigs} algorithm, outputting unitigs in FASTA format, representing a \textbf{lossy} (frequency-wise) sequence summary. $\weight(\W)$ is the size of the \texttt{.fasta} file.
\end{itemize}

This comparison focuses on the raw text representation size ($\weight(\W)$) generated by the core graph traversal algorithms relative to the input sequence data size ($\weight(\M)$, proxied by the intermediate \texttt{.reads} file size) before any secondary compression like BWT or gzip. Execution time measures the core algorithm runtime for MCTR (internally timed) and total wall-clock time for \texttt{GreedyTigs}.

\textbf{Results:}

The performance comparison is detailed in Table~\ref{tab:perf_summary} and visualized in Figure~\ref{fig:perf_comparison_raw}.

\begin{table}[h!]
\centering
\caption{Performance Comparison on E. coli Datasets (Raw Compression)}
\label{tab:perf_summary}
\resizebox{\textwidth}{!}{%
\begin{tabular}{l l r r r r}
\toprule
Dataset & Algorithm             & k  & Core/Total Time (s) & Raw Output Size (bytes) & Compr. Ratio \\
\midrule
sim     & MCTR (list)           & 31 & 2.57          & 4,689,720        & 1.00x                        \\ 
sim     & MCTR (frequency)      & 31 & 7.95          & 11,037,214       & 0.42x                        \\ 
sim     & GreedyTigs (ctr)      & 31 & 0.93          & 4,829,258        & 0.97x                        \\ 
sim     & MCTR (list)           & 51 & 2.35          & 4,689,725        & 1.00x                        \\ 
sim     & MCTR (frequency)      & 51 & 7.39          & 10,351,638       & 0.45x                        \\ 
sim     & GreedyTigs (ctr)      & 51 & 1.13          & 4,829,540        & 0.97x                        \\ 
sim     & MCTR (list)           & 71 & 2.11          & 4,690,292        & 1.00x                        \\ 
sim     & MCTR (frequency)      & 71 & 6.75          & 9,395,535        & 0.50x                        \\ 
sim     & GreedyTigs (ctr)      & 71 & 1.36          & 4,830,768        & 0.97x                        \\ 
\midrule
real    & MCTR (list)           & 31 & 1.67          & 6,080,000        & 1.00x                        \\ 
real    & MCTR (frequency)      & 31 & 1.66          & 2,247,433        & 2.71x                        \\ 
real    & GreedyTigs (ctr)      & 31 & 0.08          & 423,317          & 14.36x                       \\ 
real    & MCTR (list)           & 51 & 2.07          & 6,080,000        & 1.00x                        \\ 
real    & MCTR (frequency)      & 51 & 2.27          & 3,300,271        & 1.84x                        \\ 
real    & GreedyTigs (ctr)      & 51 & 0.08          & 423,317          & 14.36x                       \\ 
real    & MCTR (list)           & 71 & 2.10          & 6,080,000        & 1.00x                        \\ 
real    & MCTR (frequency)      & 71 & 3.11          & 4,144,386        & 1.47x                        \\ 
real    & GreedyTigs (ctr)      & 71 & 0.09          & 423,317          & 14.36x                       \\ 
\bottomrule
\end{tabular}
} 
\vspace{0.5em}
\begin{minipage}{\textwidth}
\footnotesize Time for MCTR is core algorithm time; time for GreedyTigs is total wall clock. Raw Output Size $\weight(\W)$ for MCTR (frequency) includes both \texttt{.ctr} and \texttt{.cnt} files. Compression Ratio $\cmpr = \weight(\M) / \weight(\W)$ calculated vs. intermediate \texttt{.reads} file size $\weight(\M)$ (sim: 4,691,456 bytes; real: 6,080,000 bytes). Ratios $\leq 1$ indicate data expansion.
\end{minipage}
\end{table}

\begin{figure}[h!]
    \centering
    \includegraphics[width=\textwidth]{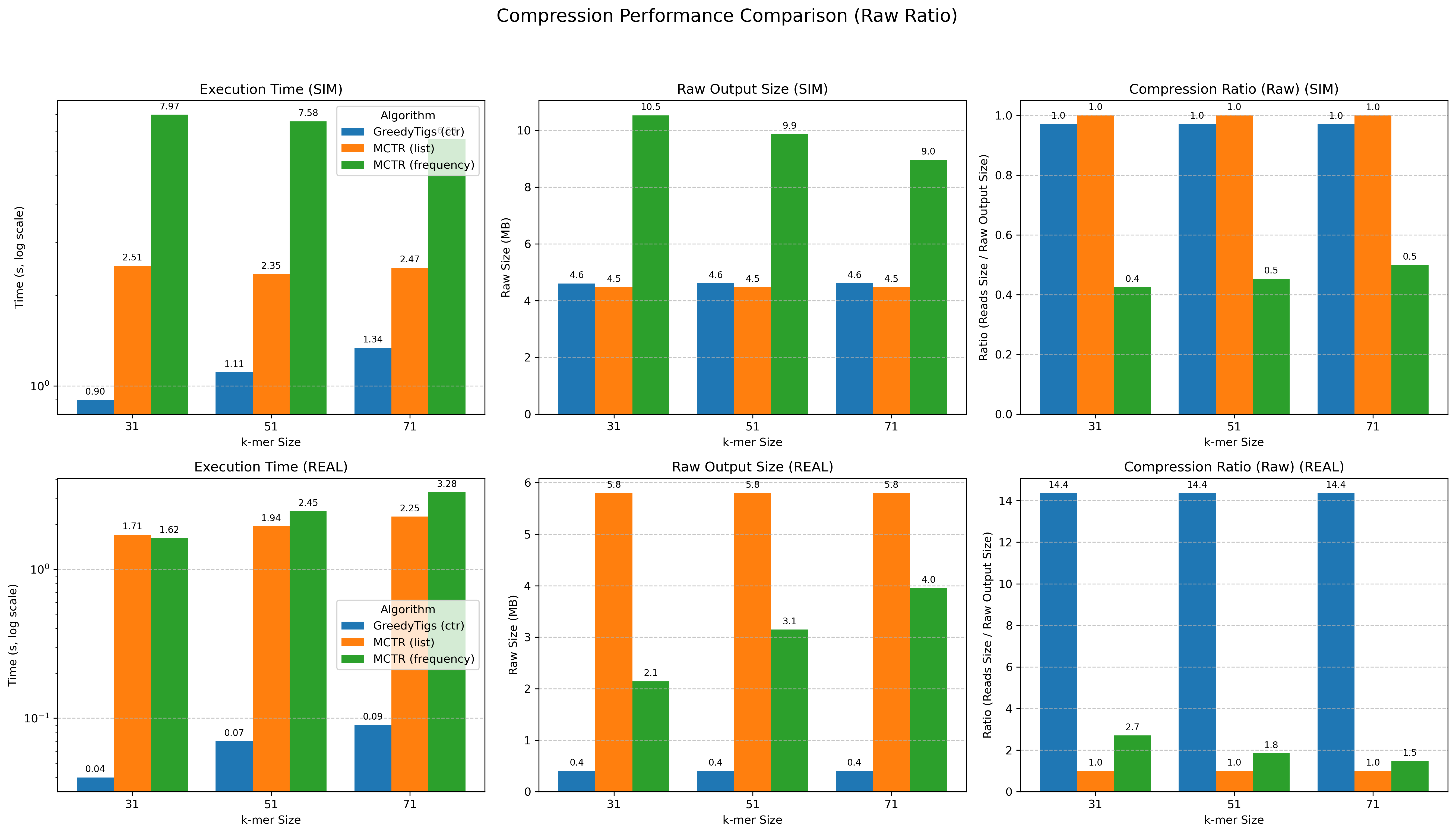}
    \caption{Performance comparison on simulated (top) and real (bottom) E. coli datasets for $k=31, 51, 71$. \textbf{Left:} Core/Total Execution Time (log scale, lower is better). \textbf{Center:} Raw Output Size ($\weight(\W)$ in MB, lower is better). \textbf{Right:} Raw Compression Ratio ($\cmpr = \weight(\M)/\weight(\W)$, higher is better). Note the data expansion ($\cmpr \leq 1$) for some methods, especially MCTR (frequency) on simulated data.}
    \label{fig:perf_comparison_raw}
\end{figure}

\textbf{Analysis:}

\begin{itemize}
    \item \textbf{Execution Time:} As shown in Figure~\ref{fig:perf_comparison_raw} (left panels) and Table~\ref{tab:perf_summary}, \texttt{GreedyTigs} is substantially faster than both MCTR variants across all conditions. While MCTR has linear time complexity, its current implementation's constant factors lead to longer runtimes on these datasets compared to the highly optimized \texttt{GreedyTigs}. \texttt{MCTR (list)} is faster than \texttt{MCTR (frequency)} due to the latter's counting overhead.

    \item \textbf{Compression Performance (Raw):} The raw compression ratio $\cmpr = \weight(\M) / \weight(\W)$ (Figure~\ref{fig:perf_comparison_raw}, right panels; Table~\ref{tab:perf_summary}) highlights the impact of representation choice and data noise. On the clean \textit{simulated data}, \texttt{MCTR (list)} achieves $\cmpr \approx 1.00$, indicating its raw output size is nearly identical to the input sequence size. \texttt{GreedyTigs} yields $\cmpr \approx 0.97$, suggesting slight expansion due to FASTA overheads. Notably, \texttt{MCTR (frequency)} results in $\cmpr < 1$ ($\approx 0.4-0.5$), indicating significant data expansion; the size overhead of the explicit count file (\texttt{.cnt}), containing mostly counts of `1' for this sparse dataset, outweighs the sequence compression in the \texttt{.ctr} file. On the \textit{real dataset}, \texttt{GreedyTigs} achieves a high raw compression ratio of $\cmpr \approx 14.4$, producing a much shorter sequence summary by likely simplifying error-prone graph regions (inherently losing frequency data). \texttt{MCTR (frequency)} achieves moderate raw compression ($\cmpr \approx 1.5-2.7$), showing benefit over the list representation as sequence compression outweighs count file overhead. \texttt{MCTR (list)} yields $\cmpr = 1.00$, indicating its lossless representation of all $k$-mers, including errors, resulted in an output size equal to the input sequence data for this dataset.
\end{itemize}

\newpage

\textbf{Discussion:}

These benchmarks clarify the trade-offs. MCTR guarantees a \textbf{lossless reconstruction of the $k$-mer multiset}. \texttt{MCTR (list)} offers a direct, space-efficient (raw size $\approx$ input size on clean data) lossless encoding. \texttt{MCTR (frequency)} provides decoupled counts but can lead to raw data expansion on sparse inputs due to count storage overhead, although it showed better raw compression than the list format on the noisy real dataset. In contrast, \texttt{GreedyTigs} optimizes for a compact \textbf{lossy sequence summary}, achieving high raw text compression ratios on noisy data by discarding frequency information and potentially simplifying sequences. MCTR remains the choice for applications demanding lossless $k$-mer spectrum preservation, while \texttt{GreedyTigs} is suited for generating concise sequence paths.

\subsection{Evaluation of BWT Enhancement}
Finally, we evaluate the effectiveness of the complete MCTR algorithm (Algorithm~\ref{alg:MCTR}), including the final BWT/RLE step applied to the compressed text representation $\W$. Using the SARS-CoV-2 datasets (SRR30994008, SRR30994007), we compared applying BWT directly to the original $k$-mer list versus applying BWT after MCTR.

\begin{figure}[h!]
    \centering
    \includegraphics[width=0.8\textwidth]{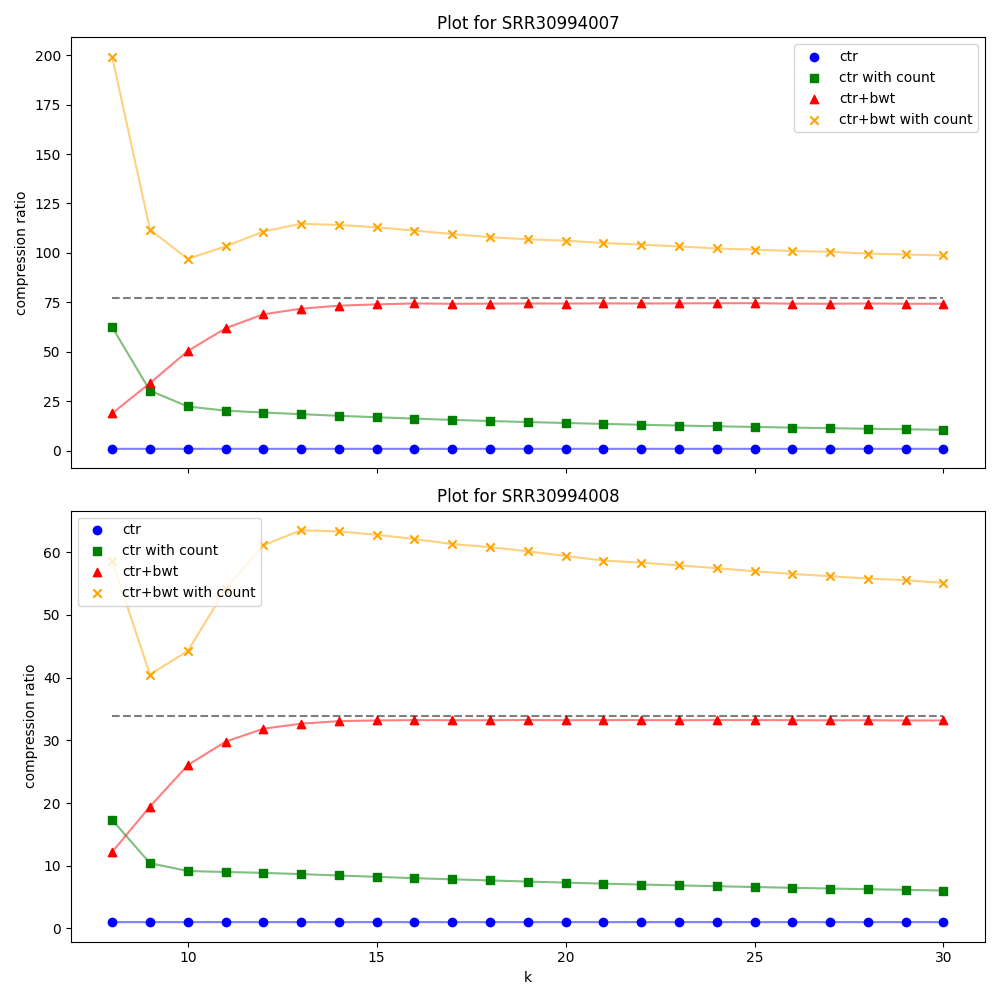}
    \caption{Comparison of compression ratios using BWT on SARS-CoV-2 datasets for $k=5$ to $k=30$. Ratios compare original $k$-mer list file size to final compressed size. ``MCTR (list) + BWT" and ``MCTR (frequency) + BWT" show the result of applying BWT to our respective text representations. ``BWT only" serves as baseline. The two-stage MCTR+BWT approach consistently yields higher lossless compression ratios.}
    \label{fig:crt_bwt}
\end{figure}

Figure~\ref{fig:crt_bwt} demonstrates that applying BWT \emph{after} MCTR consistently results in significantly higher lossless compression ratios than applying BWT alone across various $k$-values. This validates the design of our full two-stage algorithm, showing that MCTR's text representation effectively linearizes the de Bruijn graph structure, enhancing subsequent BWT compression performance.

\section{Conclusion and Future Research}\label{sec:conclusion}

We have presented MCTR, a novel approach for the \textbf{lossless} compression of $k$-mer datasets by utilizing optimal Eulerian covers of de Bruijn graphs. Our algorithm guarantees the minimum theoretical number of representing strings while preserving the complete $k$-mer multiset, including frequencies. We formally proved that MCTR achieves linear time and space complexity, ensuring scalability. The algorithm supports two lossless output formats: a direct list representation (\texttt{MCTR (list)}) and a decoupled frequency representation (\texttt{MCTR (frequency)}).

Experimental results on simulated data confirmed MCTR's adherence to theoretical properties. Benchmarks against the state-of-the-art lossy unitigging tool \texttt{greedytigs} on simulated and real E. coli data highlighted the practical implications of MCTR's lossless constraint when evaluating raw text compression ($\cmpr = \weight(\M)/\weight(\W)$). While \texttt{GreedyTigs} demonstrated superior speed, MCTR's runtime reflects its proven linear complexity, though practical speed is influenced by implementation factors. Regarding raw compression, on clean simulated data, \texttt{MCTR (list)} produced an output size nearly identical to the input sequence data ($\cmpr \approx 1.0$), while \texttt{MCTR (frequency)} resulted in data expansion ($\cmpr < 1$) due to the overhead of explicitly storing mostly unique counts. On noisy real data, \texttt{GreedyTigs} achieved high raw compression ($\cmpr \approx 14$), consistent with its goal of producing a concise (but frequency-lossy) sequence summary likely by simplifying error-prone regions. In contrast, \texttt{MCTR (list)} showed no raw compression ($\cmpr \approx 1.0$), faithfully representing all $k$-mers including errors, while \texttt{MCTR (frequency)} achieved moderate raw compression ($\cmpr \approx 1.5-2.7$), suggesting its structure can be beneficial on noisier datasets despite the count overhead. These results underscore that MCTR optimizes for lossless fidelity, which may not always translate to minimal raw text size compared to lossy methods, especially on noisy inputs.

Furthermore, evaluating the full MCTR pipeline including BWT confirmed its effectiveness for enhanced \textit{lossless} compression, significantly outperforming BWT applied directly to the original $k$-mers (Figure~\ref{fig:crt_bwt}).

In conclusion, MCTR provides a theoretically sound, efficient (in terms of asymptotic complexity), and rigorously validated \textbf{lossless} compression method for $k$-mer multisets. It serves as a crucial tool for applications where the integrity of the complete $k$-mer spectrum, including frequencies, is paramount, establishing a valuable baseline for lossless $k$-mer representation even if lossy methods achieve higher raw text or gzipped compression ratios through information reduction.

Future research directions include optimizing the MCTR implementation to reduce constant factors in runtime (e.g., through parallelization, optimized data structures). Investigating lossless error-aware graph modifications prior to Eulerization could enhance the raw text compressibility on noisy datasets without sacrificing data integrity. Applying MCTR to larger and more complex datasets (e.g., metagenomes) remains important. Finally, exploring direct bioinformatics analyses on MCTR's compressed representations could further streamline genomic data workflows.

\section*{Acknowledgments} The authors are thankful to the anonymous referees for providing valuable feedback and suggestions. The first author's work was sponsored by Natural Science Foundation of Chongqing, China (CSTB2025NSCQ-LZX0057, CSTB2025NSCQ-LZX0073), and also supported by the China Scholarship Council (CSC) (202308500094) and by the Doctoral Research  Project of Chongqing Normal University (21XLB020). The fourth author's work was supported by the research project of the Sobolev Institute of Mathematics (project FWNF-2022-0019).

\bibliographystyle{plain}
\bibliography{ref}

\end{document}